\documentclass[11pt]{article}
\usepackage{amsmath,amsfonts,amsthm,amssymb, mathtools}
\usepackage{mathrsfs, graphicx,color,latexsym, tikz, calc,
}
\usepackage{tikz-cd}
\usepackage{graphicx}
\usepackage{epstopdf}
\usepackage{enumerate}
\usepackage{caption}
\usepackage{stmaryrd}
\usepackage{hyperref}
\usepackage{xcolor}

\hypersetup{
    colorlinks=true,
    citecolor=blue,
    linkcolor=blue,
    urlcolor=blue
}

\usetikzlibrary{shadows}
\usetikzlibrary{patterns,arrows,decorations.pathreplacing}
\numberwithin{equation}{section}
\newtheorem{theorem}{Theorem}[section]
\newtheorem{lemma}[theorem]{Lemma}

\newtheorem{conjecture}[theorem]{Conjecture}

\newtheorem{remark}[theorem]{Remark}

\allowdisplaybreaks[4]

\date{}
\title{\bf \Large On a conjecture regarding the product version of the Hilton-Milner theorem
\author{{\small Xucheng Bu$^a$\, \ Lihua Feng$^a$\, \ Zejun Huang$^{b,}$\footnote{Corresponding author.
\newline{\hspace*{5mm}Email addresses:} buxcmath@163.com(X. Bu), fenglh@163.com(L. Feng), zejunhuang@szu.edu.cn(Z. Huang), lulugdmath@163.com(L. Lu), jv01065499zai@163.com(Q. Wang).
}\, \ Lu Lu$^a$\, \ Qifan Wang$^a$}\\[2mm]
\footnotesize $^a$School of Mathematics and Statistics, HNP-LAMA, Central South University\\
\footnotesize Changsha, Hunan, 410083, China\\
\footnotesize $^b$School of Mathematical Sciences, Shenzhen University\\
\footnotesize Shenzhen, 518060, China}}

\begin{document}
	\maketitle
	\begin{abstract}
		Recently, Frankl and Wang considered a product version of the classical Hilton--Milner theorem. They conjectured that, if $\mathcal{F} \subset \binom{[n]}{k}$ and $\mathcal{G} \subset \binom{[n]}{\ell}$ are non-trivial cross-intersecting families with $n \geq 2k > 2\ell \geq 4$, the maximum of $|\mathcal{F}||\mathcal{G}|$ is attained by the natural Hilton--Milner-type configurations.
In this paper, we present two main results concerning this conjecture. Firstly, we show that the conjecture does not hold in general. By introducing a  two-center construction,
we prove that for every fixed integer $\ell \geq 3$ and all sufficiently large $k$, the conjecture is false in a linear range $2k+1 \leq n \leq (c_\ell - \epsilon)k$ for any $0 < \epsilon < c_\ell - 2$, where $c_\ell > 2$ is an explicit constant.
 Secondly, we prove that   the conjecture holds when $n > 100\ell k^2$ and $3 \leq \ell < k$, and we completely characterize the extremal families. Our proofs rely on   the size of minimal covers and analyzing the structural properties of $2$-cover graphs.\\[1mm]

\end{abstract}
	
\section{Introduction}\label{sec:introduction}

Let $[n]=\{1,2,\ldots,n\}$ be the standard ground set. For an integer $r \ge 0$, we denote by $\binom{[n]}{r}$ the family of all $r$-element subsets of $[n]$. A family $\mathcal{F}\subset 2^{[n]}$ is called \emph{intersecting} if $F\cap F'\ne \emptyset$ for any $F,F'\in \mathcal{F}$.
The study of intersecting families forms the cornerstone of extremal set theory, pioneered by the classical Erd\H{o}s--Ko--Rado (EKR) theorem \cite{E61}.
\begin{theorem} [Erd\H{o}s--Ko--Rado \cite{E61}]
For $n \ge 2k$, any intersecting family $\mathcal{F} \subset \binom{[n]}{k}$ has size at most $\binom{n-1}{k-1}$. Furthermore, when $n > 2k$, equality holds if and only if $\mathcal{F}$ is a  {full star} $\mathcal{S}_x := \{F \in \binom{[n]}{k} : x \in F\}$ for some $x \in [n]$.
\end{theorem}

 Any subfamily of a full star is simply called a \emph{star}, which is often referred to as a \emph{trivial} intersecting family. A natural question arises when one prohibits the extremal configurations of the EKR theorem. An intersecting family $\mathcal{F}$ is called \emph{non-trivial} if it is not a star, or equivalently, $\bigcap_{F\in\mathcal{F}}F=\emptyset$. In 1967, Hilton and Milner \cite{HM67} determined the maximum size of a non-trivial intersecting family. They obtained
\begin{theorem}  [Hilton--Milner \cite{HM67}]\label{HM67}
For $n > 2k \ge 4$, a non-trivial intersecting family $\mathcal{F}\subset \binom{[n]}{k}$ has size at most
$$
h(n,k) = \binom{n-1}{k-1} - \binom{n-k-1}{k-1} + 1.
$$
Moreover, for $k\ge 4$, equality holds if and only if $\mathcal{F}$ is isomorphic to the so-called Hilton--Milner family:
$$
\mathcal{H}(n,k) \coloneqq \left\{ H \in \binom{[n]}{k} : 1 \in H,\ H \cap \{2,3,\ldots,k+1\} \neq \emptyset \right\} \cup \Big\{ \{2,3,\ldots,k+1\} \Big\}.
$$
\end{theorem}

Another direction in extremal set theory is to investigate multiple families. Two families $\mathcal{F}, \mathcal{G} \subset 2^{[n]}$ are called \emph{cross-intersecting} if $F \cap G \neq \emptyset$ for all $F \in \mathcal{F}$ and $G \in \mathcal{G}$. Similarly, a pair of cross-intersecting families $\mathcal{F}, \mathcal{G}$ is called \emph{non-trivial} if $\bigcap_{F\in\mathcal{F}}F=\emptyset$ and $\bigcap_{G\in\mathcal{G}}G=\emptyset$. For cross-intersecting families, bounding the sum and the product of their sizes has attracted extensive attention \cite{HW77, FT92, BF22, F24}. In this paper, we focus on the product versions.

For the uniform setting, Pyber \cite{P86} (see Frankl and Kupavskii \cite{FK17} for an elegant short proof) proved
  \begin{theorem} [Pyber \cite{P86}, Frankl--Kupavskii \cite{FK17}] \label{p86}
   If $\mathcal{F}, \mathcal{G} \subset \binom{[n]}{k}$ are cross-intersecting and $n \ge 2k \ge 4$, then $|\mathcal{F}||\mathcal{G}| \le \binom{n-1}{k-1}^2$. For $n>2k$, the maximum is attained if and only if both families are full stars centered at the same element.
\end{theorem}
\noindent  Pyber \cite{P86} also considered the unequal-uniformity generalization, showing that for $\mathcal{F} \subset \binom{[n]}{k}$ and $\mathcal{G} \subset \binom{[n]}{\ell}$ with $k \ge \ell$,  $|\mathcal{F}||\mathcal{G}|\leq \binom{n-1}{k-1}\binom{n-1}{\ell-1}$ provided $n \ge 2k+\ell-2$. This range was later improved to  $n \ge 2k$ and  $n \ge 2\ell$ by Matsumoto and Tokushige \cite{MT89}.

Recently, along the lines of Theorems \ref{HM67} and \ref{p86},
 Frankl and Wang \cite{FW23}  studied  the product analog of  non-trivial  cross-intersecting families.  They obtained
\begin{theorem} [Frankl--Wang \cite{FW23}]
 For non-trivial cross-intersecting families $\mathcal{F}, \mathcal{G} \subset \binom{[n]}{k}$, $|\mathcal{F}||\mathcal{G}| \le h(n,k)^2$ holds for $n \ge 4k$ and $k \ge 8$.
 \end{theorem}
\noindent    They subsequently improved the range to $n \ge 2k+1$ for $k \ge 8$ in \cite{FW26}. Frankl \cite{F26} also investigated the $3$-graph version. Recently, Huang \cite{H26} successfully resolved the uniform case for all $n \ge 2k$ and $k \ge 3$.

To formulate the conjecture in the unequal-uniformity setting, we need the following families.
 Following the symbols of Frankl and Wang \cite{FW26}, let $X \in \binom{[2,n]}{k}$ and $Y \in \binom{[2,n]}{\ell}$  such that $X \cap Y \neq \emptyset$. We define the corresponding families as follows:
$$
\mathcal{H}_k(X,Y) \coloneqq \{X\} \cup \left\{ F\in \binom{[n]}{k} : 1\in F,\ F\cap Y\ne\emptyset \right\},
$$
which has size
$$
M_k(n,k,\ell) \coloneqq \binom{n-1}{k-1} - \binom{n-\ell-1}{k-1} + 1.
$$
Similarly, we define
$$
\mathcal{H}_\ell(Y,X) \coloneqq \{Y\} \cup \left\{ G\in \binom{[n]}{\ell} : 1\in G,\ G\cap X\ne\emptyset \right\},
$$
which has size
$$
M_\ell(n,k,\ell) \coloneqq \binom{n-1}{\ell-1} - \binom{n-k-1}{\ell-1} +1.
$$
It is straightforward to verify that $\mathcal{H}_k(X,Y)$ and $\mathcal{H}_\ell(Y,X)$ form a non-trivial cross-intersecting pair. Based on this, Frankl and Wang \cite{FW26} proposed the following conjecture:

\begin{conjecture}[Frankl--Wang \cite{FW26}]\label{conj:FW}
Let $n \ge 2k > 2\ell \ge 4$. Suppose that $\mathcal{F} \subset \binom{[n]}{k}$ and $\mathcal{G} \subset \binom{[n]}{\ell}$ are non-trivial cross-intersecting families. Then
$$
|\mathcal{F}||\mathcal{G}| \le M_k(n,k,\ell)M_\ell(n,k,\ell).
$$
\end{conjecture}


Our first main result  Theorem \ref{thm:counter-interval}  demonstrates that Conjecture \ref{conj:FW} is  false in general. In Section \ref{Refconj:FW}, we introduce a   two-center construction. We prove that for every fixed $\ell \ge 3$ and all sufficiently large $k$,
this construction yields a strictly larger product than the Hilton--Milner-type pair for $2k+1 \le n \le (c_\ell - \epsilon)k$ for any $0 < \epsilon < c_\ell - 2$, where $c_\ell > 2$ is an explicit constant. This provides a robust family of counterexamples to the conjecture. We would like to mention that Huang \cite{H26} also provide a counterexample for $n=18$.

Our second main result is Theorem  \ref{th4.1} in Section \ref{large-n}. We confirm that when  $n > 100\ell k^2$, Conjecture \ref{conj:FW} holds. Furthermore, we characterize all extremal configurations. Our proofs  rely on estimating the sizes of minimal covers and  analyzing the structural properties of $2$-cover graphs.	

\section{A counterexample for Conjecture \ref{conj:FW}}\label{Refconj:FW}

\subsection{A two-center construction for general $n$} \label{construction}

In this subsection, we present a   two-center construction for $n \ge 2k+1$, aiming to disprove Conjecture \ref{conj:FW} in a specific range for $\ell \ge 3$, which will be presented in Subsection \ref{counterexample}.

We first  construct two set families.

Let $\ell \ge 3$ and $k > \ell$. Suppose $n \ge 2k+1$. Since $|[3,n]| = n-2 \ge 2k-1 > 2(k-1)$, we may choose two disjoint $(k-1)$-element subsets $A, B \subset [3,n]$.
In light of $A$ and $B$, we define two special $k$-element sets:
$$
F_A \coloneqq \{2\} \cup A, \qquad F_B \coloneqq \{1\} \cup B.
$$
We define our first family $\mathcal{F} \subset \binom{[n]}{k}$ as
$$
\mathcal{F} \coloneqq \left\{ F \in \binom{[n]}{k} : \{1,2\} \subset F \right\} \cup \{F_A, F_B\}.
$$
Next, we define our second family $\mathcal{G} \subset \binom{[n]}{\ell}$ by partitioning it into three parts:
\begin{align*}
	\mathcal{G}_{12} &\coloneqq \left\{ G \in \binom{[n]}{\ell} : \{1,2\} \subset G \right\}, \\[1ex]
	\mathcal{G}_1 &\coloneqq \left\{ G \in \binom{[n]}{\ell} : 1 \in G,\ 2 \notin G,\ G \cap A \neq \emptyset \right\}, \\[1ex]
	\mathcal{G}_2 &\coloneqq \left\{ G \in \binom{[n]}{\ell} : 2 \in G,\ 1 \notin G,\ G \cap B \neq \emptyset \right\}.
\end{align*}
We let $\mathcal{G} \coloneqq \mathcal{G}_{12} \cup \mathcal{G}_1 \cup \mathcal{G}_2$. We could prove that $\mathcal{F}$ and $\mathcal{G}$ have the following basic properties.

\begin{theorem}\label{thm:construction}
	The families $\mathcal{F}$ and $\mathcal{G}$ constructed above are non-trivial cross-intersecting families. Moreover, their sizes are
	$$
	|\mathcal{F}| = \binom{n-2}{k-2} + 2,
	$$
	and
	$$
	|\mathcal{G}| = \binom{n-2}{\ell-2} + 2\left( \binom{n-2}{\ell-1} - \binom{n-k-1}{\ell-1} \right).
	$$
\end{theorem}

\begin{proof}
	By construction, $\mathcal{F} \subset \binom{[n]}{k}$ and $\mathcal{G} \subset \binom{[n]}{\ell}$.
We divide the proof into three parts.
\vspace{2mm}
	
	{\bf First, we show that $\mathcal{F}$ and $\mathcal{G}$ are cross-intersecting}, i.e.,  $F \cap G \neq \emptyset$ for any $F \in \mathcal{F}$ and $G \in \mathcal{G}$.
	Observe that every member containing $\{1,2\}$ of   $\mathcal{F}$   intersects all $G \in \mathcal{G}$, this is because every set in $\mathcal{G}$ contains at least one of $1$ and $2$.
	It remains to check $F_A$ and $F_B$. Consider $F_A = \{2\} \cup A$. For any $G \in \mathcal{G}_{12} \cup \mathcal{G}_2$, we have $2 \in F_A \cap G$. For any $G \in \mathcal{G}_1$,   $G \cap A \neq \emptyset$, which implies $\emptyset\ne A\cap G\subset F_A\cap G$.
	Similarly, for $F_B = \{1\} \cup B$, we have $1 \in F_B \cap G$ whenever $G \in \mathcal{G}_{12} \cup \mathcal{G}_1$. If $G \in \mathcal{G}_2$, then $G \cap B \neq \emptyset$, yielding $F_B \cap G \supset B \cap G \neq \emptyset$. Thus, $\mathcal{F}$ and $\mathcal{G}$ are cross-intersecting.
\vspace{2mm}	

	{\bf Next, we show that they are non-trivial}. For $\mathcal{F}$, since $A$ and $B$ are disjoint subsets of $[3,n]$, we have $F_A \cap F_B = (\{2\} \cup A) \cap (\{1\} \cup B) = \emptyset$. This implies $\bigcap_{F \in \mathcal{F}} F = \emptyset$. For $\mathcal{G}$, we will show that no element of $[n]$ belongs to all members of $\mathcal{G}$. Elements $1$ and $2$ are excluded by the definitions of $\mathcal{G}_2$ and $\mathcal{G}_1$, respectively. For any $x \in A$, since $|A| = k-1 \ge \ell > \ell-1$, we can choose an $(\ell-1)$-subset $A' \subset A \setminus \{x\}$. The set $G = \{1\} \cup A'$ satisfies $1 \in G, 2 \notin G$, and $G \cap A = A' \neq \emptyset$. Hence, $G \in \mathcal{G}_1$, and it does not contain $x$.
	Similarly, for any $y \in B$, we can find a subset $B' \subset B \setminus \{y\}$ of size $\ell-1$, and the constructed set $G = \{2\} \cup B' \in \mathcal{G}_2$ does not contain $y$.
	Finally, if $z \notin \{1,2\} \cup A \cup B$, then none of the sets constructed above contains $z$. Therefore, $\bigcap_{G \in \mathcal{G}} G = \emptyset$.\vspace{2mm}
	
	 {\bf Finally, we compute their sizes}.
 The size of $\mathcal{F}$ is easy to obtain.
 For $\mathcal{G}$, the subfamilies $\mathcal{G}_{12}, \mathcal{G}_1$, and $\mathcal{G}_2$ are pairwise disjoint. Clearly, $|\mathcal{G}_{12}| = \binom{n-2}{\ell-2}$.
	To determine $|\mathcal{G}_1|$, we consider the $\binom{n-2}{\ell-1}$ subsets of size $\ell$ that contain $1$ but avoid $2$. Among these, the sets that are disjoint from $A$ must be chosen entirely from $[n] \setminus (\{1,2\} \cup A)$. Since $|A| = k-1$, there are $n-2-(k-1) = n-k-1$ such available elements. Consequently, there are $\binom{n-k-1}{\ell-1}$ sets that fail to intersect $A$. Subtracting this from the total gives
	$$
	|\mathcal{G}_1| = \binom{n-2}{\ell-1} - \binom{n-k-1}{\ell-1}.
	$$
	By similar reasoning for $B$, we obtain $|\mathcal{G}_2| = \binom{n-2}{\ell-1} - \binom{n-k-1}{\ell-1}$. Summing the sizes of these three disjoint parts yields the final formula for $|\mathcal{G}|$.
\end{proof}

	\subsection{The counterexample interval}\label{counterexample}

In this subsection, we will disprove Conjecture \ref{conj:FW} using the two-center construction presented in Subsection \ref{construction}.

Fix $\ell \ge 3$. Let $q_\ell \in (1/2, 1)$ be the unique real root of the polynomial equation
$$
1 + q + q^2 + \cdots + q^{\ell-1} = 2.
$$
The root $q_\ell$    is unique because the left-hand side of the equation is strictly increasing with respect to $q$. For $q = 1/2$, the sum is   $2 - 2^{1-\ell} < 2$, and for $q = 1$, the sum is $\ell \ge 3 > 2$.
We then define the absolute constant
$$
c_\ell \coloneqq \frac{1}{1-q_\ell}.
$$
  Note that $c_\ell > 2$ as $q_\ell > 1/2$.

\begin{theorem}\label{thm:counter-interval}
	Fix $\ell \ge 3$, and let $c_\ell$ be defined as above. For every $0 < \epsilon < c_\ell - 2$, there exists an integer $K = K(\ell, \epsilon)$ such that for every $k \ge K$ and every integer $n$ satisfying
	$
	2k+1 \le n \le (c_\ell - \epsilon)k,
	$
	the families $\mathcal{F}$ and $\mathcal{G}$ in Theorem \ref{thm:construction} satisfy
	$$
	|\mathcal{F}||\mathcal{G}| > M_k(n,k,\ell)M_\ell(n,k,\ell).
	$$
	Consequently, Conjecture \ref{conj:FW} is false for every fixed $\ell \ge 3$.
\end{theorem}

\begin{proof}
Here $O(\cdot)$ and $\Theta(\cdot)$ are with respect to $k \to \infty$.
Let
	$$
	c_k \coloneqq \frac{n}{k}.
	$$
	The given range of $n$ implies that $c_k \in [2 + 1/k, c_\ell - \epsilon]$. We evaluate the  ratios $|\mathcal{F}| / M_k(n,k,\ell)$ and $|\mathcal{G}| / M_\ell(n,k,\ell)$.

First, recall that
	\begin{equation*}\label{fromth0201}
	|\mathcal{F}| = \binom{n-2}{k-2} + 2, \qquad M_k(n,k,\ell) = \binom{n-1}{k-1} - \binom{n-\ell-1}{k-1} + 1.
\end{equation*}
	Dividing the main terms by $\binom{n-1}{k-1}$, we have
	$$
	\frac{\binom{n-2}{k-2}}{\binom{n-1}{k-1}} = \frac{k-1}{n-1} = \frac{1 - 1/k}{c_k - 1/k} = \frac{1}{c_k} + O\left(\frac{1}{k}\right),
	$$
	and
	\begin{align*}
		\frac{\binom{n-\ell-1}{k-1}}{\binom{n-1}{k-1}}
		= \prod_{i=0}^{\ell-1} \frac{n-k-i}{n-1-i}
		&= \prod_{i=0}^{\ell-1} \left( \frac{n-k}{n} \cdot \frac{1 - \frac{i}{n-k}}{1 - \frac{i+1}{n}} \right) \\
		&= \left( 1 - \frac{1}{c_k} \right)^\ell \left( 1 + O\left(\frac{1}{n}\right) \right)^\ell \\
		&= \left( 1 - \frac{1}{c_k} \right)^\ell + O\left(\frac{1}{k}\right).
	\end{align*}
Thus, we obtain
	\begin{equation} \label{eq:ratio-F}
		\frac{|\mathcal{F}|}{M_k(n,k,\ell)} = \frac{c_k^{-1}}{1 - \left( 1 - c_k^{-1} \right)^\ell} + O\left(\frac{1}{k}\right).
	\end{equation}
	
	Next, we estimate $|\mathcal{G}| / M_\ell(n,k,\ell)$. Recall that
\[|\mathcal{G}| = 2\left( \binom{n-2}{\ell-1} - \binom{n-k-1}{\ell-1} \right) + \binom{n-2}{\ell-2}, \]
and
\[M_\ell(n,k,\ell) = \binom{n-1}{\ell-1} - \binom{n-k-1}{\ell-1} + 1.\]
Let $  \Upsilon \coloneqq \binom{n-2}{\ell-1} - \binom{n-k-1}{\ell-1}$. Clearly, we have
\[\Upsilon=\frac{c_k^{\ell-1}-(c_k-1)^{\ell-1}}{(\ell-1)!}k^{\ell-1}+\Theta(k^{\ell-2}).\]
Since $c_k\in[2+1/k, c_{\ell}-\epsilon]$ and $c_k^{\ell-1}-(c_k-1)^{\ell-1}\ge 2^{\ell-1}-1>0$, we obtain
\begin{equation}\label{eq-l-1} \Upsilon=\Theta(k^{\ell-1}).\end{equation}
Note that	\begin{align*}
		|\mathcal{G}| &= 2\Upsilon + \binom{n-2}{\ell-2}, \qquad
		M_\ell(n,k,\ell)  = \Upsilon + \binom{n-2}{\ell-2} + 1.
	\end{align*}
By \eqref{eq-l-1}, we get
	\begin{equation} \label{eq:ratio-G}
		\frac{|\mathcal{G}|}{M_\ell(n,k,\ell)} = \frac{2\Upsilon + O(k^{\ell-2})}{\Upsilon + O(k^{\ell-2})} = 2 + O\left(\frac{1}{k}\right).
	\end{equation}
Combining \eqref{eq:ratio-F} and \eqref{eq:ratio-G}, the product ratio is given by
	$$
	\frac{|\mathcal{F}||\mathcal{G}|}{M_k(n,k,\ell)M_\ell(n,k,\ell)} = R_\ell(c_k) + O\left(\frac{1}{k}\right),
	$$
	where
	$$
	R_\ell(c) \coloneqq \frac{2}{c\left( 1 - \left( 1 - \frac{1}{c} \right)^\ell \right)}.
	$$
	By setting $q = 1 - 1/c$, we have $c = 1/(1-q)$, which yields
	$$
	c\left( 1 - \left( 1 - \frac{1}{c} \right)^\ell \right) = \frac{1 - q^\ell}{1 - q} = 1 + q + q^2 + \cdots + q^{\ell-1}.
	$$
	Thus, $R_\ell(c) > 1$ if and only if $1 + q + \cdots + q^{\ell-1} < 2$. By the definition of $q_\ell$ and $c_\ell$, this condition holds strictly for all $2 \le c < c_\ell$. Let $\delta=\delta(\ell,\epsilon)=R_{\ell}(c_{\ell}-\epsilon)-1$ be a positive constant. Thus, for sufficiently large $k\ge K(\ell,\epsilon)$, we get
	\[\frac{|\mathcal{F}||\mathcal{G}|}{M_k(n,k,\ell)M_\ell(n,k,\ell)} \ge 1 + \delta + O\left(\frac{1}{k}\right)>1,\]
	as desired.
	
%
\end{proof}

\begin{remark}
	The first few threshold values of $c_\ell$ are approximated below:
	\begin{center}
		\renewcommand{\arraystretch}{1.2}
		\begin{tabular}{cc}
			\hline
			$\ell$ & $c_\ell$ \\
			\hline
			$3$ & $2.6180339887\dots$ \\
			$4$ & $2.1914878839\dots$ \\
			$5$ & $2.0780950776\dots$ \\
			$6$ & $2.0352521616\dots$ \\
			\hline
		\end{tabular}
	\end{center}
	For instance, when $\ell = 3$, for every fixed $c<c_3=2.618033...$, the construction yields counterexamples for sufficiently large $k$ and all $n$ satisfying $2k+1\le n\le ck$. While the valid counterexample interval narrows as $\ell$ increases, it remains a non-empty linear interval for any fixed $\ell \ge 3$.
\end{remark}

	\section{Conjecture \ref{conj:FW} is true for large $n$} \label{large-n}

In this section we confirm that Conjecture \ref{conj:FW} is true when $n=n(k, \ell)$ is sufficiently large. Furthermore, we completely characterize the extremal families.

Our second main result is

\begin{theorem} \label{th4.1}
	Let $3 \le \ell < k$ and assume $n > 100\ell k^2$. For every pair of non-trivial cross-intersecting families
	$
	\mathcal{F} \subset \binom{[n]}{k},  \mathcal{G} \subset \binom{[n]}{\ell},
	$
	  we have
	$$
	|\mathcal{F}||\mathcal{G}| \le M_k(n,k,\ell)M_\ell(n,k,\ell).
	$$
	Moreover, equality holds if and only if, up to isomorphism, there exist subsets
	$$
	X \in \binom{[2,n]}{k}, \qquad Y \in \binom{[2,n]}{\ell} \qquad \text{with } X \cap Y \neq \emptyset,
	$$
	such that the families precisely match the Hilton--Milner-type configurations:
	\begin{align*}
		\mathcal{F} &= \{X\} \cup \left\{ F \in \binom{[n]}{k} : 1 \in F,\ F \cap Y \neq \emptyset \right\}, \\
		\mathcal{G} &= \{Y\} \cup \left\{ G \in \binom{[n]}{\ell} : 1 \in G,\ G \cap X \neq \emptyset \right\}.
	\end{align*}
\end{theorem}

\vskip 2mm

Our proof strategy is to reduce the problem to set pairs that are maximal with respect to cross-intersection. We then analyze the structure of their $2$-cover graphs and bound the higher-order minimal covers using an encoding argument. Finally, we carefully compare these upper bounds against the conjectured bound.

\subsection{Some auxiliary lemmas}
In this subsection, we always assume $3 \le \ell < k$ and   $n > 100\ell k^2$.
\vskip 2mm

If a  pair of non-trivial cross-intersecting families
	$
	\mathcal{F} \subset \binom{[n]}{k},  \mathcal{G} \subset \binom{[n]}{\ell}
	$
satisfy
	$$
		F \in \binom{[n]}{k} \text{ and } F \cap G \neq \emptyset \ \text{for all }G \in \mathcal{G} \quad \Longrightarrow \quad F \in \mathcal{F}$$
and
		$$G \in \binom{[n]}{\ell} \text{ and } G \cap F \neq \emptyset \  \text{for all } F \in \mathcal{F} \quad \Longrightarrow \quad G \in \mathcal{G},$$
then we say the pair $(\mathcal{F}, \mathcal{G})$ is {\it maximal}.

To prove Theorem \ref{th4.1}, it suffices to consider pairs $(\mathcal{F}, \mathcal{G})$ that are maximal with respect to cross-intersection, since otherwise we can add possible $k$-sets to $\mathcal{F}$ or $\ell$-sets to $\mathcal{G}$ to increase $|\mathcal{F}|| \mathcal{G}|$ while maintaining all other required properties. \vskip 2mm

{\bf In what follows, we assume that $(\mathcal{F}, \mathcal{G})$ is a maximal non-trivial cross-intersecting pair.} Now we prepare some preliminary definitions and lemmas.
\vskip 2mm

A set $T \subset [n]$ is called a \emph{cover} of a family $\mathcal{A}$ if $T \cap A \neq \emptyset$ for every $A \in \mathcal{A}$. The covering number $\tau(\mathcal{A})$ is the minimum size of a cover of $\mathcal{A}$. Since the given families $\mathcal{F}$ and $\mathcal{G}$ are non-trivial, they cannot be covered by a single element. Hence, $\tau(\mathcal{F}) \ge 2$ and $\tau(\mathcal{G}) \ge 2$.

 We encode the minimal covers of size $2$ by defining the $2$-cover graphs as
\begin{align*}
	\mathcal{P} &\coloneqq \left\{ P \in \binom{[n]}{2} : P \cap F \neq \emptyset \text{ for every } F \in \mathcal{F} \right\}, \\
	\mathcal{Q} &\coloneqq \left\{ Q \in \binom{[n]}{2} : Q \cap G \neq \emptyset \text{ for every } G \in \mathcal{G} \right\}.
\end{align*}

\begin{lemma} \label{lem4.2}
	For the graphs $\mathcal{P}$ and $\mathcal{Q}$, their maximum degrees and sizes satisfy
	\begin{equation}\label{h1}
	\Delta(\mathcal{P}) \le k, \qquad \Delta(\mathcal{Q}) \le \ell, \qquad |\mathcal{P}| \le k^2, \qquad |\mathcal{Q}| \le \ell^2,
	\end{equation}
	and $\mathcal{P}, \mathcal{Q}$ are cross-intersecting.
\end{lemma}
\begin{proof}
	We only prove the inequalities in \eqref{h1} for $\mathcal{P}$, the arguments for $\mathcal{Q}$ are similar.

For any $x \in [n]$, since $\mathcal{F}$ is non-trivial, there exists $F_x \in \mathcal{F}$ such that $x \notin F_x$. Since any edge $\{x, y\} \in \mathcal{P}$ must intersect $F_x$, we have $y \in F_x$. Thus $|F_x|=k$ leads to $\Delta(\mathcal{P}) \le k$. Now fix an arbitrary $F_0 \in \mathcal{F}$. Since every edge of $\mathcal{P}$ must intersect $F_0$, summing the degrees of the vertices in $F_0$ we can deduce $|\mathcal{P}| \le k  |F_0| = k^2$.
	
	To see that $\mathcal{P}$ and $\mathcal{Q}$ are cross-intersecting, suppose for contradiction that there exist $P \in \mathcal{P}$ and $Q \in \mathcal{Q}$ with $P \cap Q = \emptyset$. Since $P$ is a cover of $\mathcal{F}$, any $\ell$-set containing $P$ automatically intersects every member of $\mathcal{F}$. Because $n > 100\ell k^2$ guarantees $n - 4 \ge \ell - 2$, we can easily extend $P$ to an $\ell$-set $G'$ that avoids $Q$ (i.e., $P \subset G'$ and $G' \cap Q = \emptyset$). By the maximality of $\mathcal{G}$, we must have $G' \in \mathcal{G}$. However, $Q \in \mathcal{Q}$ is a cover of $\mathcal{G}$, which forces $Q \cap G' \neq \emptyset$, a contradiction.
\end{proof}

\begin{lemma} \label{lem4.3}
	Let $\mathcal{P}, \mathcal{Q} \subset \binom{[n]}{2}$ be cross-intersecting graphs. Suppose $k > \ell \ge 3$, $\Delta(\mathcal{P}) \le k$, and $\Delta(\mathcal{Q}) \le \ell$. Then either $\mathcal{P}$ and $\mathcal{Q}$ are    stars sharing the same center with $|\mathcal{P}| = k$ and $|\mathcal{Q}| = \ell$, or they satisfy
	$$
	|\mathcal{P}||\mathcal{Q}| \le k\ell - \frac{\ell}{2}.
	$$
\end{lemma}
\begin{proof}
	If  one of $\mathcal{P}$ and $\mathcal{Q}$  is empty, the upper bound holds trivially. Thus, we assume both are non-empty. We proceed by analyzing the structure of $\mathcal{P}$.
\vspace{2mm}	
	
	 \textbf{Case 1: $\mathcal{P}$ is a star.} Let $v$ be its center. If $|\mathcal{P}| \ge 3$, by Lemma \ref{lem4.2}, any edge of $\mathcal{Q}$ must intersect all edges of $\mathcal{P}$, which forces every edge of $\mathcal{Q}$ to contain $v$. Thus, $\mathcal{Q}$ is also a star centered at $v$. It follows that
 $$|\mathcal{P}||\mathcal{Q}| = \Delta(\mathcal{P})\Delta(\mathcal{Q}) \le k\ell.$$
 If $|\mathcal{P}||\mathcal{Q}| = k\ell$, then  $|\mathcal{P}| = k$ and $|\mathcal{Q}| = \ell$, which is the exceptional equality case.
 If $|\mathcal{P}| < k$ or $|\mathcal{Q}| < \ell$, the product drops by at least $\min(k, \ell) = \ell$, satisfying   $|\mathcal{P}||\mathcal{Q}| \le k\ell - \ell \le k\ell - \ell/2$.

If $|\mathcal{P}| = 1$, then $\mathcal{Q}$ meets a single fixed edge. Since $\Delta(\mathcal{Q}) \le \ell$, we get $|\mathcal{Q}| \le 2\ell$. Recalling that $k > \ell \ge 3$, we have $|\mathcal{P}||\mathcal{Q}| \le 2\ell < k\ell - \ell/2$.

	If $|\mathcal{P}| = 2$, say $\mathcal{P} = \{\{v,u\}, \{v,w\}\}$, then any edge of $\mathcal{Q}$ either contains $v$ or is exactly $\{u,w\}$, which leads to  $|\mathcal{Q}| \le \ell + 1$. Again, we have $|\mathcal{P}||\mathcal{Q}| \le 2(\ell + 1) < k\ell - \ell/2$.

\vspace{2mm}

 \textbf{Case 2: $\mathcal{P}$ is intersecting but not a star.} Then $\mathcal{P}$ must be a triangle. Since any edge that intersects all three edges of a triangle must itself belong to the triangle, $\mathcal{Q}$ is a subgraph of the same triangle, which implies $|\mathcal{P}||\mathcal{Q}| \le 9<k\ell - \ell/2$, as $k>\ell\ge 3$.

\vspace{2mm}
\textbf{Case 3: $\mathcal{P}$ contains two disjoint edges.} Let $e = \{a,b\},f = \{c,d\}\in \mathcal{P}$. Since every edge of $\mathcal{Q}$  meets both $e$ and $f$,   $\mathcal{Q}$ is a subgraph of $K_{2,2} = \{\{a,c\}, \{a,d\}, \{b,c\}, \{b,d\}\}$. Let $m = |\mathcal{Q}| \in \{1,2,3,4\}$.
	If $m=1$, then $\mathcal{P}$ meets one edge, so $|\mathcal{P}| \le 2k$, and $|\mathcal{P}||\mathcal{Q}| \le 2k < k\ell - \ell/2$.
	If $m=2$ and $\mathcal{Q}$'s edges are disjoint, then $\mathcal{P}$ must meet both. Thus we have $|\mathcal{P}| \le 4$, which leads to $|\mathcal{P}||\mathcal{Q}| \le 8 < k\ell - \ell/2$.
	If $m=2$ and $\mathcal{Q}$'s edges share a vertex, then $\mathcal{P}$ must contain that vertex or be the remaining edge. Thus we have $|\mathcal{P}| \le k+1$, which implies $|\mathcal{P}||\mathcal{Q}| \le 2(k+1) \le k\ell - \ell/2$.
	If $m \ge 3$, it is easy to check that $\mathcal{P}$ contains at most $3$ edges in the case $m=3$ and 2 edges in the case $m=4$, which forces $|\mathcal{P}||\mathcal{Q}| < k\ell - \ell/2$.	
\end{proof}

\begin{lemma} \label{lem4.4}
	Let $\mathcal{A} \subset \binom{[n]}{r}$ be non-empty. For every $s \ge 1$, the number of minimal covers (under inclusion) of $\mathcal{A}$ of size $s$ is at most $r^s$.
\end{lemma}
\begin{proof}
	Fix an arbitrary linear order on $[n]$. We define a deterministic encoding for any minimal cover $T$ of size $s$. For any set $S \subset [n]$ that does not cover $\mathcal{A}$, let $A(S)$ be a canonical choice of a set in $\mathcal{A}$ such that $A(S) \cap S = \emptyset$.
	
	Start with $S_0 = \emptyset$. Since $T$ covers $\mathcal{A}$, $T \cap A(S_0) \neq \emptyset$. Let $x_1$ be the first element of $T \cap A(S_0)$ under our ordering, and set $S_1 = \{x_1\}$.
For each subsequent step $1 \le i < s$, $S_i$ is a proper subset of $T$. Because $T$ is  minimal (under inclusion),
 $S_i$ cannot cover $\mathcal{A}$. Thus, $A(S_i)$ is well-defined and disjoint from $S_i$, but it must intersect $T$. Let $x_{i+1}$ be the first element of $(T \setminus S_i) \cap A(S_i)$, and set $S_{i+1} = S_i \cup \{x_{i+1}\}$.
	
	This process generates a sequence $(x_1, x_2, \ldots, x_s)$ containing exactly the elements of $T$. At step $i$, the element $x_{i+1}$ is chosen from the $r$-element set $A(S_i)$. This implies there are at most $r^s$ possible sequences. Since the mapping from $T$ to its generating sequence is injective, the bound holds.
\end{proof}

\begin{lemma} \label{lem:anchored-covers}
	Let $\mathcal{A} \subset \binom{[n]}{r}$ be non-empty, and let $E \subset [n]$ have size $2$. For every $s \ge 2$, the number of  minimal (under inclusion) covers of $\mathcal{A}$ of size $s$ that intersect $E$ is at most $2r^{s-1}$.
\end{lemma}
\begin{proof}
	It suffices to prove that for any fixed $z \in [n]$, the number of minimal covers of size $s$ containing $z$ is at most $r^{s-1}$. The result then follows by summing over the two elements of $E$. We use the exact same encoding procedure as in Lemma \ref{lem4.4}, but initialized with $S_1 = \{z\}$. Since $s \ge 2$ and $T$ is minimal, $S_1$ is not a cover, allowing us to select $x_2$ from $A(S_1)$, and so forth. There are $s-1$ remaining steps, each offering at most $r$ choices, yielding $r^{s-1}$ such covers.
\end{proof}

To streamline the asymptotic expressions, we denote
$$
N_k \coloneqq \binom{n-2}{k-2}, \qquad N_\ell \coloneqq \binom{n-2}{\ell-2}, \qquad p \coloneqq |\mathcal{P}|, \qquad q \coloneqq |\mathcal{Q}|.
$$

\begin{lemma} \label{lem4.6}
	Suppose $n \ge 4k^2$. Then the following strict estimates hold.
	\begin{enumerate}
		\item[(i)] If $\mathcal{P} \neq \emptyset$, then $|\mathcal{F}| \le \left(q + \frac{8k\ell^2}{n}\right)N_k$.
		If $\mathcal{P} = \emptyset$, then $|\mathcal{F}| \le \left(q + \frac{4k\ell^3}{n}\right)N_k$.
		\item[(ii)] If $\mathcal{Q} \neq \emptyset$, then $|\mathcal{G}| \le \left(p + \frac{8k^2\ell}{n}\right)N_\ell$.
		If $\mathcal{Q} = \emptyset$, then $|\mathcal{G}| \le \left(p + \frac{4k^3\ell}{n}\right)N_\ell$.
	\end{enumerate}
\end{lemma}
\begin{proof}
	We only present the proof of (i), as the same argument works for (ii) by swapping the roles of $k, \ell, \mathcal{P}$ and $\mathcal{Q}$. Since every $F \in \mathcal{F}$ is a cover of $\mathcal{G}$,
 $F$ contains at least one  minimal (under inclusion) cover $T$ of $\mathcal{G}$. Because $\mathcal{G}$ is non-trivial, no singleton covers it, which implies $|T| \ge 2$.
	
	If $|T| = 2$, then $T \in \mathcal{Q}$. Hence, there are exactly $q=|\mathcal{Q}|$ choices for $T$, and each belongs to at most $\binom{n-2}{k-2} = N_k$ members of $\mathcal{F}$. This contributes at most $q N_k$ to the size of $\mathcal{F}$.
	
	If $|T| = r \ge 3$, Lemma \ref{lem4.4} bounds the number of minimal $r$-covers of $\mathcal{G}$ by $\ell^r$. Moreover, each such cover is contained in at most $\binom{n-r}{k-r}$ members of $\mathcal{F}$.

If $\mathcal{P} = \emptyset$, summing over all possible sizes we have
	$$
	|\mathcal{F}| \le qN_k + \sum_{r=3}^{k} \ell^r \binom{n-r}{k-r}.
	$$
	For $3 \le r \le k$, since $n \ge 4k^2$, applying the trivial bound $n-k \ge n/2$, we have
	$$
	\frac{\binom{n-r}{k-r}}{\binom{n-2}{k-2}} = \frac{(k-2)\cdots(k-r+1)}{(n-2)\cdots(n-r+1)} \le \left(\frac{2k}{n}\right)^{r-2}.
	$$
	Denote by $x = 2k\ell/n $. Then $x\le 1/2$ and the summation term is bounded by
	$$
	N_k \sum_{r=3}^{k} \ell^r \left(\frac{2k}{n}\right)^{r-2} = N_k \ell^2 \sum_{j=1}^{k-2} x^j \le N_k \ell^2 (2x) = \frac{4k\ell^3}{n}N_k.
	$$
	
	When $\mathcal{P} \neq \emptyset$, choose an arbitrary edge $P_0 \in \mathcal{P}$.
We claim that any minimal cover $T$ of $\mathcal{G}$ of size $3 \le |T| \le k$ must intersect $P_0$. In fact, if $T \cap P_0 = \emptyset$, since $n-2 \ge k$, we could trivially extend $T$ to a $k$-set $F'$ avoiding $P_0$ and
 covering $\mathcal{G}$, so by maximality, $F' \in \mathcal{F}$, which contradicts the fact that $P_0 \in \mathcal{P}$ covers $\mathcal{F}$.
  By Lemma \ref{lem:anchored-covers}, there are at most $2\ell^{r-1}$ such minimal $r$-covers. Again, summing all possible sizes we have
	$$
	|\mathcal{F}| \le qN_k + \sum_{r=3}^{k} 2\ell^{r-1} \binom{n-r}{k-r} \le qN_k + N_k (2\ell) \sum_{j=1}^{k-2} x^j \le qN_k + \frac{8k\ell^2}{n}N_k.
	$$
\end{proof}

\begin{lemma} \label{lem4.7} 
	For $n \ge 2k\ell$, we have
	$$
	M_k(n,k,\ell) \ge \ell\left(1 - \frac{2k\ell}{n}\right)N_k \quad \text{and} \quad M_\ell(n,k,\ell) \ge k\left(1 - \frac{2k\ell}{n}\right)N_\ell.
	$$
	Consequently, their product satisfies
	$$
	M_k(n,k,\ell)M_\ell(n,k,\ell) \ge \left(k\ell - \frac{4k^2\ell^2}{n}\right)N_k N_\ell.
	$$
\end{lemma}
\begin{proof}
	By counting only the $k$-subsets that contain $1$ and exactly one specific point of $Y \subset [2,n]$ (where $|Y| = \ell$), we obtain the immediate lower bound $M_k(n,k,\ell) \ge \ell\binom{n-\ell-1}{k-2}$. Its ratio to $N_k$ is
	$$
	\frac{\binom{n-\ell-1}{k-2}}{\binom{n-2}{k-2}} = \prod_{i=0}^{k-3} \left( 1 - \frac{\ell-1}{n-2-i} \right) \ge \prod_{i=0}^{k-3} \left( 1 - \frac{\ell-1}{n-k} \right).
	$$
	Applying Weierstrass's product inequality $\prod_i (1-a_i) \ge 1 - \sum_i a_i$, we get
	$$
	\frac{\binom{n-\ell-1}{k-2}}{\binom{n-2}{k-2}} \ge 1 - \frac{(k-2)(\ell-1)}{n-k} \ge 1 - \frac{k\ell}{n-k} \ge 1 - \frac{2k\ell}{n},
	$$
	where the last inequality holds because $n \ge 2k$. This establishes the bound for $M_k$.

 A symmetric calculation yields the bound for $M_\ell$. Multiplying these lower bounds we have
	$$
	M_k M_\ell \ge k\ell \left(1 - \frac{2k\ell}{n}\right)^2 N_k N_\ell \ge k\ell \left(1 - \frac{4k\ell}{n}\right) N_k N_\ell = \left(k\ell - \frac{4k^2\ell^2}{n}\right)N_k N_\ell,
	$$
	where we used the standard inequality $(1-x)^2 \ge 1-2x$.
\end{proof}

\subsection{Proof of Theorem \ref{th4.1}}
Now we are ready to present the proof of Theorem \ref{th4.1}.
\vspace{2mm}

\begin{proof}[Proof of Theorem \ref{th4.1}]
	We assume $3 \le \ell < k$, $n > 100\ell k^2$ and $(\mathcal{F}, \mathcal{G})$ is a maximal non-trivial cross-intersecting pair. Let $p = |\mathcal{P}|$ and $q = |\mathcal{Q}|$.  By Lemma \ref{lem4.2}, we have
$$\Delta(\mathcal{P}) \le k,\quad \Delta(\mathcal{Q}) \le \ell, \quad p\le k^2,\quad q\le \ell^2,$$
and $\mathcal{P}$,$\mathcal{Q}$ are cross-intersecting. We  distinguish  three cases.
\vspace{2mm}

	 \textbf{Case 1: Both $\mathcal{P}$ and $\mathcal{Q}$ are non-empty, but not common-center stars with $|\mathcal{P}|=k$ and $|\mathcal{Q}|=\ell$.}
By Lemma \ref{lem4.3}, we have  $ pq \le k\ell - \ell/2.$ Choose an arbitrary $Q_0 \in \mathcal{Q}$, then every $P \in \mathcal{P}$ intersects $Q_0$, which implies $p \le 2k$, as $\Delta(\mathcal{P}) \le k$. Similarly, $q \le 2\ell$.
	Setting $\alpha \coloneqq 8k\ell^2 / n$ and $\beta \coloneqq 8k^2\ell / n$, applying Lemma \ref{lem4.6} we have $$|\mathcal{F}| \le (q+\alpha)N_k\quad \text{and} \quad |\mathcal{G}| \le (p+\beta)N_\ell. $$Thus,
	\begin{align*}
		\frac{|\mathcal{F}||\mathcal{G}|}{N_k N_\ell} &\le pq + q\beta + p\alpha + \alpha\beta \\
		&\le \left(k\ell - \frac{\ell}{2}\right) + (2\ell)\left(\frac{8k^2\ell}{n}\right) + (2k)\left(\frac{8k\ell^2}{n}\right) + \frac{64k^3\ell^3}{n^2} \\
		&= k\ell - \frac{\ell}{2} + \frac{32k^2\ell^2}{n} + \frac{64k^3\ell^3}{n^2}.
	\end{align*}
	We claim this is strictly less than the lower bound $k\ell - 4k^2\ell^2/n$ obtained in Lemma \ref{lem4.7}, which is equivalent to the inequality
	$$
	\frac{36k^2\ell^2}{n} + \frac{64k^3\ell^3}{n^2} < \frac{\ell}{2}.
	$$
In fact, since $n > 100\ell k^2$, we have
	$$
	\frac{36k^2\ell^2}{n} + \frac{64k^3\ell^3}{n^2} <  \frac{36k^2\ell^2}{100\ell k^2} + \frac{64k^3\ell^3}{10000\ell^2 k^4}< 0.36\ell+ 0.0064\frac{\ell}{k}<0.5\ell.
	$$ Thus, we have $|\mathcal{F}||\mathcal{G}| < M_k M_\ell$.
	\vspace{2mm}

	 \textbf{Case 2: At least one of $\mathcal{P}$ and $\mathcal{Q}$ is empty.}
	Assume first that $\mathcal{Q} = \emptyset$ and $\mathcal{P} \neq \emptyset$. Then $q=0$. Applying Lemma \ref{lem4.2} and Lemma \ref{lem4.6}, we have
	\begin{equation*}
	\frac{|\mathcal{F}||\mathcal{G}|}{N_k N_\ell} \le \frac{8k\ell^2}{n} \left( p + \frac{4k^3\ell}{n} \right) \le \frac{8k\ell^2}{n} \left( k^2 + \frac{4k^3\ell}{n} \right).
	\end{equation*}
	We want to ensure this is less than $k\ell - 4k^2\ell^2/n$, which is equivalent to
	\begin{equation}\label{eqh2}
	\frac{8k^2\ell}{n} + \frac{32k^3\ell^2}{n^2} + \frac{4k\ell}{n} < 1
	\end{equation}
by  dividing both sides with  $k\ell$.
	Recall that  $n > 100\ell k^2$, one may verify \eqref{eqh2} directly.

For the subcase  $\mathcal{P} = \emptyset$ and $\mathcal{Q} \neq \emptyset$, we can use a  similar argument to obtain a strict inequality.
	If $\mathcal{P} = \emptyset$ and $\mathcal{Q} = \emptyset$, by Lemma \ref{lem4.6} we have $$|\mathcal{F}||\mathcal{G}| / (N_k N_\ell) \le 16k^4\ell^4 / n^2. $$
After dividing by $k\ell$, the bound $16k^3\ell^3 / n^2 + 4k\ell / n < 0.0016 (\ell/k) + 0.04/k < 1$ guarantees $|\mathcal{F}||\mathcal{G}| < M_k M_\ell$.
	
	\vspace{2mm}
	 \textbf{Case 3:  $\mathcal{P}$ and $\mathcal{Q}$ are common-center stars with $|\mathcal{P}|=k$ and $|\mathcal{Q}|=\ell$.}
	  Relabelling the ground set, we may assume the center is $1$. Then there exist sets $X, Y \subset [2,n]$ with $|X|=k$ and $|Y|=\ell$ such that $\mathcal{P} = \{\{1,x\} : x \in X\}$ and $\mathcal{Q} = \{\{1,y\} : y \in Y\}$.
	
	Because $\mathcal{F}$ is non-trivial, there exists some $F_0 \in \mathcal{F}$ that avoids $1$. Since every edge $\{1,x\} \in \mathcal{P}$ is a cover of $\mathcal{F}$, $F_0$ must intersect $\{1,x\}$ for all $x \in X$. Since $1 \notin F_0$, we must have $x \in F_0$,
implying $X \subset F_0$. Since $|X| = k = |F_0|$, we uniquely determine that $F_0 = X$, which means $X$ is the  only  set in $\mathcal{F}$ avoiding $1$.
 Similarly, there is a unique set $G_0 \in \mathcal{G}$ avoiding $1$, which must be $Y$. Because $X \in \mathcal{F}$, $Y \in \mathcal{G}$ and the families are cross-intersecting, it forces $X \cap Y \neq \emptyset$.
	
	Now consider any $F \in \mathcal{F}$ containing $1$. It must intersect $Y \in \mathcal{G}$. Hence,
	$$
	\mathcal{F} \subset \{X\} \cup \left\{ F \in \binom{[n]}{k} : 1 \in F,\ F \cap Y \neq \emptyset \right\} = \mathcal{H}_k(X, Y).
	$$
	Similarly,
	$$
	\mathcal{G} \subset \{Y\} \cup \left\{ G \in \binom{[n]}{\ell} : 1 \in G,\ G \cap X \neq \emptyset \right\} = \mathcal{H}_\ell(Y, X).
	$$
	These inclusions demonstrate that $|\mathcal{F}| \le M_k(n,k,\ell)$ and $|\mathcal{G}| \le M_\ell(n,k,\ell)$, establishing the conjectured upper bound $|\mathcal{F}||\mathcal{G}| \le M_k M_\ell$. Furthermore, equality holds if and only if both families precisely coincide with these Hilton--Milner-type configurations. This completes the proof of Theorem \ref{th4.1}.
\end{proof}

	\section*{Declaration of competing interest}
	The authors declare that they have no conflicts of interests to this paper.
	
	\section*{Data availability}
	No data was used for the research described in the paper.
	
	\section*{Acknowledgement}
	 Lihua Feng is supported by the NSFC (Nos. 12271527 and 12471022) and NSF of Qinghai Province (No. 2025-ZJ-902T). Lu Lu is supported by the NSFC (No. 12371362).


\begin{thebibliography}{99}{\small
\bibitem{BF22}
P. Borg, C. Feghali, The maximum sum of sizes of cross-intersecting families of subsets of a set, {\em Discrete Math.} 345(11) (2022) 112981.
\bibitem{E61}
P. Erd{\H{o}}s, C. Ko, R. Rado, Intersection theorems for systems of finite sets, {\em Quart. J. Math. Oxford Ser. (2)} 12(2) (1961) 313--320.
\bibitem{F24}
P. Frankl, On the maximum of the sum of the sizes of non-trivial cross-intersecting families, {\em Combinatorica} 44(1) (2024) 15--35.
\bibitem{F26}
P. Frankl, The maximum of the product of non-trivial cross-intersecting 3-graphs, {\em Combinatorics and Number Theory} 15(1) (2025)  1--8.
\bibitem{FK17}
P. Frankl, A. Kupavskii, A size-sensitive inequality for cross-intersecting families, {\em European J. Combin.} 62 (2017) 263--271.
\bibitem{FT92}
P. Frankl, N. Tokushige, Some best possible inequalities concerning cross-intersecting families, {\em J. Combin. Theory Ser. A} 61(1) (1992) 87--97.
\bibitem{FW23}
P. Frankl, J. Wang, A product version of the Hilton-Milner Theorem, {\em J. Combin. Theory Ser. A} 200 (2023) 105791.
\bibitem{FW26}
P. Frankl, J. Wang, A product version of the Hilton-Milner Theorem II, (2026) arXiv:2605.09246.
\bibitem{HW77}A. J. W. Hilton, An intersection theorem for a collection of families of subsets of a finite set, {\em J. London Math. Soc.} 2(3) (1977) 369--376.
\bibitem{HM67}
A. J. W. Hilton, E. C. Milner, Some intersection theorems for systems of finite sets, {\em Quart. J. Math.} 18(1) (1967) 369--384.
\bibitem{H26}
Y. Huang, A sharp product bound for non-trivial cross-intersecting families, (2026) arXiv:2606.23322.
\bibitem{MT89}
M. Matsumoto, N. Tokushige, The exact bound in the Erd\H{o}s-Ko-Rado theorem for cross-intersecting families, {\em J. Combin. Theory Ser. A} 52 (1989) 90--97.
\bibitem{P86}
L. Pyber, A new generalization of the Erd\H{o}s-Ko-Rado theorem, {\em J. Combin. Theory Ser. A} 43(1) (1986) 85--90.
}		
\end{thebibliography}
\end{document}